\documentclass[12pt,letterpaper]{article}
\usepackage[left=0.9in,top=0.9in,right=0.9in,bottom=0.9in]{geometry}
\usepackage{amsfonts,amsmath,amsthm,amssymb,latexsym,mathrsfs,stmaryrd}
\usepackage{hyperref}
\usepackage{theoremref}

\theoremstyle{theorem}
\newtheorem{theorem}{Theorem}[section]

\newtheorem{lemma}[theorem]{Lemma}

\theoremstyle{definition}

\theoremstyle{remark}

\numberwithin{equation}{section}

\usepackage{amsfonts,amsmath,amsthm,amssymb,latexsym,mathrsfs,stmaryrd} \usepackage{mathtools}

\usepackage{tikz}
\usetikzlibrary{matrix,arrows,decorations.pathmorphing}
\usepackage{tikz-cd}
\tikzset{commutative diagrams/.cd}

\newcommand\C{{\mathbb C}}

\newcommand\Fq{{{\mathbb F}_q}}

\DeclareMathOperator{\Hom}{Hom}
\DeclareMathOperator{\End}{End}

\DeclareMathOperator{\Spec}{Spec}

\DeclareMathOperator{\GL}{GL}

\DeclareMathOperator{\Mat}{Mat}

\newcommand\subeq{\subseteq}

\def\acts{\curvearrowright}

\DeclarePairedDelimiter{\abs}{\lvert}{\rvert}

\DeclarePairedDelimiter{\set}{\{}{\}}

\DeclarePairedDelimiter{\parens}{\lparen}{\rparen}

\DeclarePairedDelimiter\floor{\lfloor}{\rfloor}

\usepackage{enumitem}
\setenumerate[1]{label=\textup{(\alph*)}}
\setenumerate[2]{label=\textup{(\roman*)}}

\DeclareMathOperator{\ord}{ord}
\newcommand{\F}{\mathbb F}
\newcommand{\fS}{\mathcal S}
\newcommand{\cP}{\mathcal P}
\newcommand{\g}{\mathfrak g}

\begin{document}
\title{Counting on the variety of modules over the quantum plane}
\author{Yifeng Huang\footnote{Dept.\ of Mathematics, University of Michigan. Email Address: {\tt huangyf@umich.edu}. The author is supported by the Research Training Grant (RTG) in Number Theory and Representation Theory at the University of Michigan.}}
\date{\today}

\bibliographystyle{abbrv}

\maketitle
\begin{abstract}
Let $\zeta$ be a fixed nonzero element in a finite field $\Fq$ with $q$ elements. In this article, we count the number of pairs $(A,B)$ of $n\times n$ matrices over $\Fq$ satisfying $AB=\zeta BA$ by giving a generating function. This generalizes a generating function of Feit and Fine that counts pairs of commuting matrices. Our result can be also viewed as the point count of the variety of modules over the quantum plane $xy=\zeta yx$, whose geometry was described by Chen and Lu. 
\end{abstract}

\section{Introduction}
\subsection{Main results}
Fix a nonzero element $\zeta$ in $\Fq$, the finite field with $q$ elements. Let $\ord(\zeta)$ denote the smallest positive integer $m$ such that $\zeta^m=1$ in $\Fq$. We define the set of $\Fq$-points of the \emph{$\zeta$-commuting variety} to be
\begin{equation}
K_{\zeta,n}(\Fq):=\{(A,B)\in \Mat_n(\Fq)\times \Mat_n(\Fq): AB=\zeta BA\}.
\end{equation}

The $\zeta$-commuting variety $K_{\zeta,n}$ can be viewed as the variety of $n$-dimensional modules over the algebra of the quantum plane, namely, the noncommutative associate algebra in variables $x$ and $y$ such that $xy=\zeta yx$. The geometry of the $\zeta$-commuting variety was studied by Chen and Lu \cite{chenlu2019}, where explicit descriptions of its irreducible components and of a GIT quotient were given. The combinatorics of the $\zeta$-commuting has also been studied when $\zeta=1$: Feit and Fine \cite{feitfine1960pairs} gave an explicit formula for the point count of the \emph{commuting} variety (namely, $K_{1,n}$) over a finite field, and Bryan and Morrison \cite{bryanmorrison2015motivic} proved that the ``same'' formula computes the motivic class of the commuting variety (over $\C$) in the Grothendieck ring of varieties. 

The focus of this paper is to count the cardinality of $K_{\zeta,n}(\Fq)$ for $\zeta$ in general. As a special case, the cardinality of $K_{1,n}(\Fq)$, the set of pairs of commuting matrices, was determined by Feit and Fine \cite{feitfine1960pairs} in the form of a generating function. We give a generating function for $\abs{K_{\zeta,n}(\Fq)}$ that generalizes the $\zeta=1$ case. 

\begin{theorem}\label{thm:A}
Let $m=\ord(\zeta)$; in other words, $\zeta$ is a primitive $m$-th root of unity of $\Fq$. We have the following identity of power series in $x$:
\begin{equation}
\sum_{n=0}^\infty \frac{\abs{K_{\zeta,n}(\Fq)}}{(q^n-1)(q^n-q)\dots(q^n-q^{n-1})} x^n = \prod_{i=1}^\infty F_m(x^i;q),
\end{equation}
where
\begin{equation}
F_m(x;q):=\dfrac{1-x^m}{(1-x)(1-x^mq)} \cdot \dfrac{1}{(1-x)(1-xq^{-1})(1-xq^{-2})\dots}.
\end{equation}
\end{theorem}

We note that $\abs{K_{\zeta,n}(\Fq)}$ only depends on the order $m$ of $\zeta$. When $m=1$, we recover the generating function given by Feit and Fine. 

Theorem \ref{thm:A} is a direct consequence of the following result, which in itself can be viewed as a refinement of Theorem \ref{thm:A}. We define
\begin{equation}
U_{\zeta,n}(\Fq):=\{(A,B)\in \Mat_n(\Fq)\times \Mat_n(\Fq): AB=\zeta BA,\text{ $A$ nonsingular}\},
\end{equation}
and
\begin{equation}
N_{\zeta,n}(\Fq):=\{(A,B)\in \Mat_n(\Fq)\times \Mat_n(\Fq): AB=\zeta BA,\text{ $A$ nilpotent}\}.
\end{equation}

When $\zeta=-1$, the variety $N_{-1,n}$ is the semi-nilpotent anti-commuting variety, whose irreducible components and their dimensions are explicitly described by Chen and Wang \cite{chenwang2020}. 

For brevity reason, we put $\abs{\GL_n(\Fq)}$ in place of $(q^n-1)(q^n-q)\dots(q^n-q^{n-1})$ in the formulas below, noting that $\abs{\GL_n(\Fq)}=(q^n-1)(q^n-q)\dots(q^n-q^{n-1})$.

\begin{theorem}\label{thm:B}
Let $m=\ord(\zeta)$. We have the following identities of power series in $x$:
\begin{enumerate}
\item
\begin{equation}
\sum_{n=0}^\infty \frac{\abs{K_{\zeta,n}(\Fq)}}{\abs{\GL_n(\Fq)}} x^n = \parens*{\sum_{n=0}^\infty \frac{\abs{U_{\zeta,n}(\Fq)}}{\abs{\GL_n(\Fq)}} x^n} \parens*{\sum_{n=0}^\infty \frac{\abs{N_{\zeta,n}(\Fq)}}{\abs{\GL_n(\Fq)}} x^n}
\end{equation}

\item 
\begin{equation}
\sum_{n=0}^\infty \frac{\abs{U_{\zeta,n}(\Fq)}}{\abs{\GL_n(\Fq)}} x^n = \prod_{i=1}^\infty G_m(x^i;q),
\end{equation}
where
\begin{equation}
G_m(x;q):=\frac{1-x^m}{(1-x)(1-x^mq)}.
\end{equation}

\item 
\begin{equation}
\sum_{n=0}^\infty \frac{\abs{N_{\zeta,n}(\Fq)}}{\abs{\GL_n(\Fq)}} x^n = \prod_{i=1}^\infty H(x^i;q),
\end{equation}
where
\begin{equation}
H(x;q):=\frac{1}{(1-x)(1-xq^{-1})(1-xq^{-2})\dots}.
\end{equation}
\end{enumerate}
\end{theorem}

Using Theorem \ref{thm:B}(a), Theorem \ref{thm:A} follows from the observation $F_m(x;q)=G_m(x;q)H(x;q)$.

Note that Theorem \ref{thm:B}(c) implies that $\abs{N_{\zeta,n}(\Fq)}$ does not depend on $m$ or $\zeta$, as long as $\zeta\neq 0$. In particular, $\abs{N_{\zeta,n}(\Fq)}$ always equals $\abs{N_{1,n}(\Fq)}$, which is known to Feit and Fine. Therefore, the nontrivial dependence of $\abs{K_{\zeta,n} (\Fq)}$ on $\zeta$ stems purely from that of $\abs{U_{\zeta,n}(\Fq)}$. 

\subsection{History and related work}
An important starting case in the study of varieties of modules is the commuting variety $K_{1,n}=\set{(A,B):A,B\in \Mat_n, AB=BA}$. The commuting variety over $\C$ was shown to be irreducible by Gerstenhaber \cite{gerstenhaber1961} and Motzkin and Taussky \cite{motzkintaussky1955}. Its point count was given by Feit and Fine \cite{feitfine1960pairs}. This result was reproved by Bryan and Morrison \cite{bryanmorrison2015motivic} from the perspective of motivic Donaldson--Thomas theory. 

The commuting variety can be viewed in the context of Lie algebras. Let $(\g,[\cdot,\cdot])$ be a Lie algebra over an algebraically closed field. Define the \emph{commuting variety} of $\g$ as
\begin{equation}
C(\g):=\set{(x,y)\in \g\times \g: [x,y]=0},
\end{equation}
then $K_{1,n}$ is the commuting variety of the Lie algebra of $n$ by $n$ matrices. As a generalization of the irreducibility result of $K_{1,n}$, Richardson \cite{richardson1979commuting} showed that the commuting variety of any reductive Lie algebra over $\C$ is irreducible. Levy \cite{levy2019commuting} extended this result to positive characteristic under mild restrictions on the Lie algebra. On the combinatorics side, Fulman and Guralnick \cite{fulmanguralnick2018} generalized the point-count result of Feit and Fine to commuting varieties of unitary groups and of odd characteristic sympletic groups. We also point out some papers that relate counting problems in Lie algebras to maximal tori of Lie groups; see \cite{fjrw2017generating} and \cite{lehrer1992rational}.

The focus of this paper, the $\zeta$-commuting variety $K_{\zeta,n}$, is another generalization of the commuting variety $K_{1,n}$. When $\zeta=-1$, we get the anti-commuting variety, whose geometry over $\C$ was studied by Chen and Wang \cite{chenwang2020}. They gave explicit descriptions of the irreducible components of $K_{-1,n}$ and of several variants. The above work was extended to general $\zeta$ by Chen and Lu \cite{chenlu2019}. It is worth noting that $K_{\zeta,n}$ is not irreducible unless $\zeta=1$. The main contribution of our paper is the point count of $K_{\zeta,n}$.

The point count of $K_{\zeta,n}$ can also be viewed as statistical information on the classification of modules over the quantum plane. In specific, $K_{\zeta,n}(\Fq)$ is the weighted count of isomorphism classes of $n$-dimensional modules over the quantum plane over $\Fq$, with weight inversely proportional to the size of the automorphism group. The classification is studied by Bavula \cite{andrewspartitions}. 

\section{Proof of Theorem \ref{thm:B}(a)}
We recall that Theorem \ref{thm:B}(a) claims that $\abs{K_{\zeta,n}(\Fq)}$ for all $n$ can be recovered from $\abs{U_{\zeta,n}(\Fq)}$ and $\abs{N_{\zeta,n}(\Fq)}$ for all $n$. We start by proving a decomposition lemma, following the approach of Feit and Fine \cite{feitfine1960pairs}.

Let $V$ be an $n$-dimensional vector space over any field, then by Fitting's lemma (see for instance \cite[p.\ 113]{jacobsonbasic2}), for any linear map $A\in \End(V)$, there is a unique decomposition $V=K_A\oplus I_A$ such that $A(K_A)\subeq K_A, A(I_A)\subeq I_A$, $A|_{K_A}$ is nilpotent, and $A|_{I_A}$ is nonsingular. 

\begin{lemma}\label{lem:decomp}
Fix a linear map $A\in \End(V)$ and a nonzero scalar $\zeta$. Then a linear map $B\in \End(V)$ satisfies $AB=\zeta BA$ if and only if
\begin{enumerate}
\item $B(K_A)\subeq K_A$, $B(I_A)\subeq I_A$.
\item $A|_{K_A}B|_{K_A}=\zeta B|_{K_A}A|_{K_A}$, $A|_{I_A}B|_{I_A}=\zeta B|_{I_A}A|_{I_A}$.
\end{enumerate}
\end{lemma}
\begin{proof}
Having the decomposition $V=K_A\oplus I_A$, any linear map $X\in \End(V)$ can be written as a matrix
\begin{equation}
X=\begin{bmatrix}
X_1 & X_2\\
X_3 & X_4
\end{bmatrix}, 
\begin{array}{l}
X_1\in \End(K_A), X_2\in \Hom(I_A,K_A),\\
X_3\in \Hom(K_A,I_A), X_4\in \End(I_A).
\end{array} 
\end{equation}

Then we have
\begin{equation}
A=\begin{bmatrix}
N & 0\\
0 & U
\end{bmatrix}
\end{equation}
where $N\in \End(K_A)$ is nilpotent and $U\in \End(I_A)$ is nonsingular. For an arbitrary $B=\begin{bmatrix}
B_1 & B_2\\
B_3 & B_4
\end{bmatrix}$, the equation $AB=\zeta BA$ is equivalent to
\begin{equation}
\begin{cases}
NB_1 = \zeta B_1 N,\\
NB_2 = \zeta B_2 U,\\
UB_3 = \zeta B_3 N,\\
UB_4 = \zeta B_4 U.
\end{cases}
\end{equation}

We note that $B_2$ must be zero. Suppose not, since $N$ is nilpotent, there exists an integer $r\geq 0$ such that $N^r B_2\neq 0$ but $N^{r+1} B_2=0$. The second equation gives $N^{r+1}B_2=\zeta N^r B_2 U$. The left-hand side is zero, while the right-hand side is nonzero because $\zeta$ is a nonzero scalar and $U$ is nonsingular. This yields a contradiction.

A similar argument shows that $B_3=0$, completing the proof of the lemma.
\end{proof}

Let $V=\Fq^n$. To choose $A,B\in \End(V)$ with $AB=\zeta BA$, because of Lemma \ref{lem:decomp}, it suffices to choose a decomposition $V=K\oplus I$, and then choose $A_K, B_K\in \End(K), A_I, B_I\in \End(I)$ such that $A_K$ is nilpotent, $A_K B_K=\zeta B_K A_K$, $A_I$ is nonsingular, and $A_I B_I = \zeta B_I A_I$. We have
\begin{equation}
\abs{K_{\zeta,n}(\Fq)}=\sum_{s+t=n} h(s,t) \abs{N_{\zeta,s}(\Fq)} \abs{U_{\zeta,t}(\Fq)},
\end{equation}
where $h(s,t)$ is the number of ordered pairs $(K,I)$ of subspaces of $V$ such that $\dim K=s, \dim I=t$. 

It is noted by Feit and Fine \cite[Equation (5)]{feitfine1960pairs} that
\begin{equation}
h(s,t)=\frac{\abs{\GL_n(\Fq)}}{\abs{GL_s(\Fq)}\abs{GL_t(\Fq)}}.
\end{equation}

It follows that
\begin{align}
\sum_{n=0}^\infty \frac{\abs{K_{\zeta,n}(\Fq)}}{\abs{\GL_n(\Fq)}} x^n &= \sum_{n=0}^\infty \sum_{s+t=n} \frac{\abs{\GL_n(\Fq)}}{\abs{GL_s(\Fq)}\abs{GL_t(\Fq)}} \abs{N_{\zeta,s}(\Fq)} \abs{U_{\zeta,t}(\Fq)} \frac{1}{\abs{\GL_n(\Fq)}} x^n  \\
&=\sum_{s,t\geq 0} \frac{\abs{N_{\zeta,s}(\Fq)}}{\abs{GL_s(\Fq)}}\frac{\abs{U_{\zeta,t}(\Fq)}}{\abs{GL_t(\Fq)}} x^{s+t}\\
&=\parens*{\sum_{s=0}^\infty \frac{\abs{N_{\zeta,s}(\Fq)}}{\abs{GL_s(\Fq)} }x^s} \parens*{\sum_{t=0}^\infty \frac{\abs{U_{\zeta,t}(\Fq)}}{\abs{GL_t(\Fq)}}x^t},
\end{align}
completing the proof of Theorem \ref{thm:B}(a).

\section{Proof of Theorem \ref{thm:B}(b)}
Recall that the goal of Theorem \ref{thm:B}(b) is to determine $\abs{U_{\zeta,n}(\Fq)}$, namely, to enumerate the paris of matrices $(A,B)\in \Mat_n(\Fq)\times \Mat_n(\Fq)$ such that $AB=\zeta BA$ and $A$ is nonsingular. To do so, following the approach of Feit and Fine, let $\beta$ be a similarity class of $n\times n$ matrices. By a standard orbit-stabilizer argument, for $B$ in $\beta$, the number of nonsingular matrices $A$ such that $ABA^{-1}=\zeta B$ is either $\abs{\GL_n(\Fq)}/\abs{\beta}$ or zero. Moreover, this number is not zero if and only if $B$ is similar to $\zeta B$. We now give a sufficient and necessary condition for it in terms of $\beta$. 

We recall that each class $\beta$ corresponds to a unique rational canonical form. It is characterized by an $n$-dimensional module $M_\beta$ of the polynomial ring $\Fq[t]$. Such a module can be uniquely expressed in the form of
\begin{equation}
M_\beta = \frac{\Fq[t]}{(g_1(t))} \oplus \frac{\Fq[t]}{(g_2(t))} \oplus \dots \oplus \frac{\Fq[t]}{(g_r(t))}
\end{equation}
for monic polynomials $g_1,\dots,g_r$ such that $g_i$ divides $g_{i+1}$ for all $1\leq i\leq r-1$. For a positive integer $m$, we say a monic polynomial $g$ to be in $\cP_m$ if $g(t)=t^b G(t^m)$ for some nonnegative integer $b$ and monic polynomial $G$. For example, a polynomial is in $\cP_2$ if it is either even or odd. 

\begin{lemma}\label{lem:similar}
Let $B$ be an $n\times n$ matrix over any field, and let $\zeta$ be an $m$-th root of unity. Then $B$ is similar to $\zeta B$ if and only if the polynomials $g_1,\dots,g_r$ associated to the rational canonical form of $B$ are in $\cP_m$. 
\end{lemma}
\begin{proof}
We denote the ground field by $\F$. An endomorphism $B$ over a vector space $V$ determines a module over the polynomial ring $\F[t]$ by letting $t\cdot v=Bv$ for $v\in V$. We denote this $\F[t]$-module by $(B \acts V)$. The isomorphism class of this $\F[t]$-module determines the rational canonical form of $B$.

Let $g_1,\dots,g_h$ be the polynomials associated to the rational canonical form of $B$. Then
\begin{equation}
(B \acts V)\cong \frac{\F[t]}{(g_1(t))} \oplus \frac{\F[t]}{(g_2(t))} \oplus \dots \oplus \frac{\F[t]}{(g_r(t))}.
\end{equation}

We now compute $(\zeta B\acts V)$. We have
\begin{align}
(\zeta B\acts V) &\cong (\zeta t\acts M_B)\\
&\cong \bigoplus_{i=1}^r \parens*{ \zeta t \acts \frac{\F[t]}{(g_i(t))} }\\
&\cong \bigoplus_{i=1}^r \frac{\F[t]}{(g_i(\zeta^{-1} t))},
\end{align}
where the last isomorphism follows from (a): the action of $\zeta t$ on $\dfrac{\F[t]}{(g_i(t))}$ is cyclic, and (b): the polynomial $x\mapsto g_i(\zeta^{-1} x)$ is a minimal polynomial for $\zeta t$ acting on $\dfrac{\F[t]}{(g_i(t))}$. 

Hence, $B$ is similar to $\zeta B$ if and only if
\begin{equation}
\bigoplus_{i=1}^r \frac{\F[t]}{(g_i(t))}\cong \bigoplus_{i=1}^r \frac{\F[t]}{(g_i(\zeta^{-1} t))}
\end{equation}
as $\F[t]$-modules. Since $g_i(t)$ divides $g_{i+1}(t)$ for all $i$, we have  that $g_i(\zeta^{-1}t)$ divides $g_{i+1}(\zeta^{-1}t)$ as well. By the uniqueness statement about the polynomials associated to the rational canonical form, for each $i$, the \emph{monic} polynomials $g_i(t)$ and $\zeta^{\deg g_i} g_i(\zeta^{-1}t)$ must be equal. Write $g_i(t)=t^d+c_1 t^{d-1}+\dots+c_{d-1} t + c_d$, then $\zeta^d g_i(\zeta^{-1}t)= t^d + \zeta c_1 t^{d-1}+\dots +\zeta^{d-1} c_{d-1} t + \zeta^d c_d$. Since $\zeta$ is an $m$-th root of unity, we observe that $g_i(t)=\zeta^d g_i(\zeta^{-1}t)$ if and only if $c_j=0$ for all $j$ not divisible by $m$. This is equivalent to saying that $g_i(t)$ is in $\cP_m$. 
\end{proof}

Let $\fS_{\zeta,n}(\Fq)$ denote the set of similarity classes $\beta$ of $n\times n$ matrices over $\Fq$ such that some (equivalently, every) matrix $B$ in $\beta$ is similar to $\zeta B$. We have
\begin{align}
\abs{U_{\zeta,n}(\Fq)} &= \sum_{B\in \Mat_n(\Fq)} \abs{\set{A\in \GL_n(\Fq): ABA^{-1}=\zeta B}} \\
&= \sum_{\beta} \sum_{B\in \beta} \abs{\set{A\in \GL_n(\Fq): ABA^{-1}=\zeta B}} \\
&= \sum_{\beta\in \fS_{\zeta,n}(\Fq)} \abs{\beta} \frac{\abs{\GL_n(\Fq)}}{\abs{\beta}} + \sum_{\beta\notin \fS_{\zeta,n}(\Fq)} 0 \\
&= \abs{\GL_n(\Fq)} \abs{\fS_{\zeta,n}(\Fq)}.
\end{align}

We now count $\abs{\fS_{\zeta,n}(\Fq)}$. By Lemma \ref{lem:similar}, a similarity class in $\fS_{\zeta,n}(\Fq)$ is characterized by monic polynomials $g_1,g_2,\dots,g_r$ in $\cP_m$ such that every polynomial divides the next. Let $h_i=g_{r+1-i}/g_{r-i}$ for $1\leq i\leq t$, where $g_0=1$. It is easily checked from the definition of $\cP_m$ that $g_1,\dots,g_r$ are all in $\cP_m$ if and only if $h_1,\dots,h_r$ are all in $\cP_m$. Let $b_i=\deg h_i$. The only restriction on the monic $h_i$ is that $h_i$ is in $\cP_m$ and that $\sum_{i=1}^r ib_i=n$. We observe the important fact that the number of monic polynomials in $\cP_m$ of degree $b_i$ is $q^{\floor{b_i/m}}$. Hence to give $g_1,\dots,g_r$, we first choose $(b_i)_{i\geq 1}$ such that $\sum ib_i=n$, and then independently choose $h_i$ in $\cP_m$ of degree $b_i$. It follows that
\begin{equation}
\abs{\fS_{\zeta,n}(\Fq)}=\sum_{\substack{b_i\geq 0\\ \sum ib_i=n}} q^{\floor{b_i/m}}.
\end{equation}

Therefore, 
\begin{align}
\sum_{n=0}^\infty \frac{ \abs{U_{\zeta,n}(\Fq)} }{ \abs{\GL_n(\Fq)} } x^n &= \sum_{n=0}^\infty \abs{\fS_{\zeta,n}(\Fq)} x^n \\
&= \sum_{n\geq 0} \sum_{\substack{ b_i\geq 0 \\ \sum ib_i=n}} q^{\floor{b_i/m}} x^n \\
&= \sum_{b_1,b_2,\dots\geq 0} q^{\floor{b_i/m}} x^{\sum ib_i} \\
&= \prod_{i=1}^\infty \sum_{b=0}^\infty q^{\floor{b/m}} (x^i)^{b}.
\end{align}

By writing $b=km+l$ with $0\leq l<m$, we get 
\begin{align}
\sum_{b=0}^\infty q^{\floor{b/m}} x^{b} &= \sum_{l=0}^{m-1} \sum_{k=0}^\infty q^k x^{km+l} \\
&= \sum_{l=0}^{m-1} \frac{x^l}{1-qx^m}\\
&= \frac{1+x+\dots+x^{m-1}}{1-qx^m}\\
&= \frac{1-x^m}{(1-x)(1-qx^m)}.
\end{align}

Hence, if we define $G_m(x;q)=\dfrac{1-x^m}{(1-x)(1-qx^m)}$, then we have
\begin{equation}
\sum_{n=0}^\infty \frac{ \abs{U_{\zeta,n}(\Fq)} }{ \abs{\GL_n(\Fq)} } x^n = \prod_{i=1}^\infty G_m(x^i;q),
\end{equation}
finishing the proof of Theorem \ref{thm:B}(b).

\section{Proof of Theorem \ref{thm:B}(c)}
We follow the idea of Fine and Herstein \cite{fineherstein1958} to determine $\abs{N_{\zeta,n}(\Fq)}$, namely, the number of matrix pairs $(A,B)\in \Mat_n(\Fq)\times \Mat_n(\Fq)$ such that $AB=\zeta BA$ and $A$ is nilpotent. In fact, we will show that the situation is completely the same as the case $\zeta=1$ studied in \cite{fineherstein1958}. 

Associate to each similarity class of $n$ by $n$ nilpotent matrices a partition $\pi$ of $n$:
\begin{equation}
\pi: n=a_1 \cdot 1 + a_2 \cdot 2 + \dots,
\end{equation}
so that a representative of the similarity class associated to $\pi$ is given by
\begin{equation}
A_\pi=\begin{bmatrix}
0_{a_1} & • & • & • & • & • & • \\ 
• & 0_{a_2} & 1_{a_2} & • & • & • & • \\ 
• & • & 0_{a_2} & • & • & • & • \\ 
• & • & • & 0_{a_3} & 1_{a_3} & • & • \\ 
• & • & • & • & 0_{a_3} & 1_{a_3} & • \\ 
• & • & • & • & • & 0_{a_3} & • \\ 
• & • & • & • & • & • & \ddots
\end{bmatrix},
\end{equation}
where $0_a$ and $1_a$ denote the $a$ by $a$ zero matrix and the $a$ by $a$ identity matrix, respectively. 

Let $\alpha(\pi)$ denote the similarity class associated to $\pi$. Since the number of matrices $B$ such that $AB=\zeta BA$ only depends on the similarity class of $A$, we have
\begin{equation}
\abs{N_{\zeta,n}(\Fq)} = \sum_{\pi \vdash n} \abs{\alpha(\pi)} \abs{\set{ B\in \Mat_n(\Fq): A_\pi B = \zeta B A_\pi }}.
\end{equation}

For any fixed scalar $\zeta\neq 0$, it is elementary to check that $A_\pi B = \zeta B A_\pi$ if and only if $B$ is of the following form:
\begin{equation}
\begin{bmatrix}
\begin{array}{c|cc|ccc|cccc|c}
B_{1,1}^1 & • & B_{1,2}^1 & • & • & B_{1,3}^1 & • & • & • & B_{1,4}^1 & \cdots \\
\hline 
B_{2,1}^1 & B_{2,2}^1 & B_{2,2}^2 & • & B_{2,3}^1 & B_{2,3}^2 & • & • & B_{2,4}^1 & B_{2,4}^2 & • \\ 
• & • & \zeta B_{2,2}^1 & • & • & \zeta B_{2,3}^1 & • & • & • & \zeta B_{2,4}^1 & \cdots \\ 
\hline
B_{3,1}^1 & B_{3,2}^1 & B_{3,2}^2 & B_{3,3}^1 & B_{3,3}^2 & B_{3,3}^3 & • & B_{3,4}^1 & B_{3,4}^2 & B_{3,4}^3 & • \\ 
• & • & \zeta B_{3,2}^1 &  & \zeta B_{3,3}^1 & \zeta B_{3,3}^2 & • & • & \zeta B_{3,4}^1 & \zeta B_{3,4}^2 & \cdots \\ 
• & • & • & • & • & \zeta^2 B_{3,3}^1 & • & • & • & \zeta^2 B_{3,4}^1 & • \\ 
\hline
B_{4,1}^1 & B_{4,2}^1 & B_{4,2}^2 & B_{4,3}^1 & B_{4,3}^2 & B_{4,3}^3 & B_{4,4}^1 & B_{4,4}^2 & B_{4,4}^3 & B_{4,4}^4 & • \\ 
• & • & \zeta B_{4,2}^1 &  & \zeta B_{4,3}^1 & \zeta B_{4,3}^2 & • & \zeta B_{4,4}^1 & \zeta B_{4,4}^2 & \zeta B_{4,4}^3 & • \\ 
• & • & • & • & • & \zeta^2 B_{4,3}^1 & • & • & \zeta^2 B_{4,4}^1 & \zeta^2 B_{4,4}^2 & \cdots \\ 
• & • & • & • & • & • & • & • & • & \zeta^3 B_{4,4}^1 & • \\ 
\hline
\vdots & • & \vdots & • & \vdots & • & • & • & \vdots & • & \ddots
\end{array} 
\end{bmatrix},
\end{equation}
where each $B_{i,j}^k$ is an arbitrary $a_i$ by $a_j$ matrix, chosen independently. We note that the count $\abs{\set{ B\in \Mat_n(\Fq): A_\pi B = \zeta B A_\pi }}$ does not depend on $\zeta$. Hence,
\begin{equation}
\abs{N_{\zeta,n}(\Fq)} = \abs{N_{1,n}(\Fq)}.
\end{equation}

It is known in \cite[Equation (6)]{feitfine1960pairs} that
\begin{equation}
\abs{N_{1,n}(\Fq)} = \abs{\GL_n(\Fq)}\sum_{\pi\vdash n} \frac{1}{f(a_1)f(a_2)\dots},
\end{equation}
where $f(a):=(1-q^{-1})(1-q^{-2})\dots (1-q^{-a})$. 

Hence, 
\begin{align}
\sum_{n=0}^\infty \frac{\abs{N_{\zeta,n}(\Fq)}}{\abs{\GL_n(\Fq)}} x^n &= \sum_{n=0}^\infty \sum_{\pi\vdash n} \frac{1}{f(a_1)f(a_2)\dots} x^n \\
&= \sum_{a_1,a_2,\dots\geq 0} \frac{1}{f(a_1)f(a_2)\dots} x^{\sum ia_i} \\
&= \prod_{i=1}^\infty \sum_{a=0}^\infty \frac{1}{f(a)} (x^i)^a\\
&= \prod_{i=1}^\infty H(x^i;q),
\end{align}
where
\begin{equation}
H(x;q):=\sum_{a=0}^\infty \frac{1}{f(a)} x^a = \frac{1}{(1-x)(1-xq^{-1})(1-xq^{-2})\dots}
\end{equation}
by a classical identity due to Euler. This concludes the proof of Theorem \ref{thm:B}(c), and hence proves Theorem \ref{thm:B} and Theorem \ref{thm:A}.

\section{Discussions}
We note from the work of Bryan and Morrison \cite[\S 3.1]{bryanmorrison2015motivic} that $\abs{U_{1,n}(\Fq)}$ and $\abs{N_{1,n}(\Fq)}$ ``determine'' each other. The key ingredient is that either of the quantities above is the point count of the variety of modules over the ``commutative'' plane $\Spec \Fq[x,y]$ supported on a certain subset of closed points. A module is determined by its localizations at closed points in its support, so both $\abs{U_{1,n}(\Fq)}$ and $\abs{N_{1,n}(\Fq)}$ are determined by the point count of the variety of modules supported at a point. Since the commutative plane ``looks the same everywhere'' locally in light of the Cohen structure theorem (the complete localization of $\Fq[x,y]$ at any closed point is isomorphic to $\F[[x,y]]$ for some field extension $\F$ of $\Fq$), we can reverse the process, so that either of $\abs{U_{1,n}(\Fq)}$ and $\abs{N_{1,n}(\Fq)}$ determines the point count of the variety of modules supported at a point, and hence determines each other. 

However, for $\zeta\neq 1$, Theorem \ref{thm:B} shows that $\abs{N_{\zeta,n}(\Fq)}$ does not depend on $\zeta$ while $\abs{U_{\zeta,n}(\Fq)}$ does. Is it still possible to recover $\abs{U_{\zeta,n}(\Fq)}$ from $\abs{N_{\zeta,n}(\Fq)}$ together with the geometry of the quantum plane $xy=\zeta yx$ (which will depend on $\zeta$)?

\subsection*{Acknowledgements}
We thank Jason Bell for proposing questions that inspire this work. We thank Jeffery Lagarias and Weiqiang Wang for fruitful conversations. 

\bibliography{paper_quantum_planes.bbl}

\end{document}